\documentclass[dvipdfmx,12pt]{article}

\usepackage{amssymb,amscd,amsthm,amsmath}
\usepackage{graphicx}
\usepackage{tikz}
\usepackage{float}
\usetikzlibrary{cd}
\usetikzlibrary{intersections,calc,arrows.meta}
\usepackage{mathrsfs}

\newtheorem{theorem}{Theorem}[section]
\newtheorem{lemma}[theorem]{Lemma}

\newtheorem{cor}[theorem]{Corollary}

\theoremstyle{definition}

\newtheorem{example}[theorem]{Example}
\newtheorem{remark}[theorem]{Remark}
\newtheorem{problem}[theorem]{Problem}

\numberwithin{equation}{section}

\newcommand{\C}{\mathbb{C}}
\newcommand{\D}{\mathbb{D}}

\newcommand{\N}{\mathbb{N}}

\newcommand{\R}{\mathbb{R}}
\newcommand{\T}{\mathbb{T}}
\newcommand{\Z}{\mathbb{Z}}

\newcommand{\cW}{\mathcal{W}}

\newcommand\Tr{\operatorname{Tr}}

\newcommand\inpr[2]{\langle{#1,#2}\rangle}

\newcommand{\hf}{\hat{f}}
\newcommand{\hg}{\hat{g}}
\newcommand{\tf}{\tilde{f}}
\newcommand{\tg}{\tilde{g}}
\newcommand{\tih}{\tilde{h}}
\newcommand\ind{\operatorname{ind}}

\title{A generalized winding number formula for the Witten index of a Toeplitz operator}

\author{Masaki Izumi
\thanks{Supported in part by JSPS KAKENHI Grant Number JP20H01805}\\
Graduate School of Science \\
Kyoto University \\
Sakyo-ku, Kyoto 606-8502, Japan} 

\begin{document} 
\maketitle
\begin{abstract} 
We generalize the winding number formula for the Fredholm index of a Toeplitz operator to the Witten index. 
We also show trace formulae involving Toeplitz operators and operator monotone functions. 
\end{abstract}
\section{Introduction} 
It is well known that the Toeplitz operator $T_f$ with a continuous symbol $f\in C(\T)$ is a Fredholm operator if and only if 
$f$ has no zeros on the unit circle $\T$. 
For such $T_f$, its Fredholm index $\ind T_f$ is given by $-1$ times the winding number of $f$ 
(see, for example, \cite[Theorem 4.4.3]{A2002}). 
For $f\in C^1(\T)$, this can be paraphrased as the following formula: 
\begin{equation}\label{Find}
\ind T_f=-\frac{1}{2\pi i}\int_{0}^{2\pi}\frac{f'(t)}{f(t)}dt.
\end{equation}

A bounded operator $T\in B(H)$ on a Hilbert space $H$ is said to be almost normal if the self-commutator 
$[T^*,T]=T^*T-TT^*$ is trace class.  
The Witten index $\ind_W T$ of such $T$ is defined by the limit 
$$\ind_W T=\lim_{s\to +\infty}\Tr(e^{-sT^*T}-e^{-sTT^*}),$$
if it exists.  
If moreover $T$ is Fredholm, it is known to coincide with the Fredholm index (see \cite[Theorem 2.5]{GS1988}). 
One of the purposes of this paper is to generalize (\ref{Find}) to the Witten index by replacing 
the integral on the right-hand side with the principal value integral (see Theorem \ref{Witten} for the precise statement). 
It is worth showing such a result in view of the recent interests in the Witten index of Toeplitz operators in the literature 
(see \cite{M2020}, \cite{MSTTW2023}, \cite{ST2019} for example). 

In the celebrated work \cite{Krein1953} in 1953, Krein showed that for self-adjoint operators $A$ and $B$ with a trace class 
difference, there exists a unique integrable function $\xi_{A,B}(x)$ on $\R$, called the spectral shift function, 
such that the Krein formula 
$$\Tr(\varphi(A)-\varphi(B))=\int_\R \varphi'(x)\xi_{A,B}(x)dx$$
holds for $\varphi$ in a certain class of functions.  
It is natural to discuss the Witten index of $T$ in relation to the Krein formula for the pair $(T^*T,TT^*)$  
(see, for example, \cite{CLPS2022}, \cite{GS1988}).  
On the other hand, it is known (see \cite[p.244]{MP1989}) that the spectral shift function $\xi_{T^*T,TT^*}$ 
for an almost normal $T$ can be computed from the Pincus principal function $g_T$ as 
\begin{equation}\label{pfssf}
\xi_{T^*T,TT^*}(x)=\frac{1}{2\pi}\int_0^{2\pi}g_T(\sqrt{x}e^{i\theta})d\theta,\quad \forall x>0.
\end{equation}

The Toeplitz operator $T_f$ with a sufficiently regular symbol $f$ is almost normal and its principal function  
is determined by 
\begin{equation}\label{pf}
g_{T_f}(z)=-\ind(T_f-zI)=\frac{1}{2\pi i}\int_0^\infty \frac{f'(t)}{f(t)-z}dt,\quad \forall z\in \C\setminus f(\T),
\end{equation}
thanks to the Helton-Howe thoery (see \cite[p.244]{MP1989} ,\cite[Proposition 5.2]{HH1973}). 
This means that in principle, the Witten index of $T_f$ should be expressed by integral in terms of the symbol $f$.  
Yet, we do not pursuit this strategy to accomplish our task (cf. Remark \ref{ssf}).  
In fact, we do not need any general results in either the Krein theory or the Helton-Howe theory for our purpose. 
Instead, we directly deduce the trace formula
$$\Tr(\varphi(T_f^*T_f)-\varphi(T_fT_f^*))=\frac{1}{2\pi i}\int_\T \left(\varphi(|f(t)^2|)-\varphi(0)\right)\frac{f'(t)}{f(t)}dt$$ 
(see Theorem \ref{HHtoKF} for the precise statement) from the Helton-Howe formula 
$$\Tr([T_f,T_g])=\frac{1}{2\pi i}\int_0^{2\pi}f(t)dg(t),$$ 
which follows from an elementary argument and does not require the Helton-Howe theory. 
Since the argument in \cite[p.244]{MP1989} to describe the spectral shift function goes through the Helfton-Howe formula, 
it is more natural and efficient to deduce necessary trace formulae directly from the Helton-Howe formula than appeal  
to the Krein theory. 
One advantage of our approach is that we can treat symbols with low regularity (see Theorem \ref{AHHtoKF}).  

In Section 2, we give a proof of the Helton-Howe formula for the Toeplitz operators under a considerably mild regularity assumption, 
and deduce the above mentioned trace formula out of it.  
The contents of this section may be folklore among specialists, and our point is to present explicit formulae and also to treat symbols 
with considerably low regularity.  

As an application of the trace formula, we show in Section 3 the generalized winding number formula for the Witten index and discuss 
some examples.  
In particular, we show in Example \ref{anyv} that every real value can be realized as the Witten index of a Toeplitz operator. 

In Section 4, we extend the trace formulae shown in Section 2 to the operator monotone functions 
(Theorem \ref{pthm}, \ref{apthm}). 
Despite the fact that the Krein theory does not apply to the operator monotone functions usually, we can still do it for 
sufficiently regular symbols thanks to Peller's famous criterion \cite[Chapter 6]{P2003} 
for the Schatten-von Neumann class Hankel operators.

\section{Preliminaries}
\subsection{Notation}
We first fix the basic notation used in this paper.  
Let $H$,$H_1$ and $H_2$ be Hilbert spaces. 
We denote by $B(H_1.H_2)$ (respectively $K(H_1,H_2)$) the set of bounded (respectively compact) operators from 
$H_1$ to $H_2$. 
We write $B(H)=B(H,H)$ and $K(H)=K(H,H)$. 
For $T\in K(H_1,H_2)$, we denote by $\{s_n(T)\}_{n=1}^\infty$ the singular value sequence of $T$, 
that is, the eigenvalues of $(T^*T)^{1/2}$, counting multiplicity, in decreasing order. 
For $p>0$, we set 
$$\|T\|_p=\left(\sum_{n=1}^\infty s_n(T)^p\right)^{1/p},$$
and denote by $S_p(H_1,H_2)$ the Schatten-von Neumann $p$-class 
$$\{T\in K(H_1,H_2);\; \|T\|_p<\infty\}.$$
Our basic references for the Schatten-von Neumann classes are \cite{GK1969} and \cite{S2005}.

For a measurable function $f$ on a measure space and $0<p\leq \infty$, we denote $\|f\|_{L^p}$ by $\|f\|_p$ for simplicity. 

We denote by $\D$ the unit disk $\{z\in \C\;; |z|<1\}$. 
We denote by $dA$ the area measure of $\C$. 
Let $\T=\R/2\pi\Z$, which we often identify with the unit circle $\partial \D$. 
To define function spaces on $\T$, we always use the normalized Lebesgue measure $\frac{dt}{2\pi}$, e.g.,  
$L^p(\T)=L^p(\T,\frac{dt}{2\pi})$. 
\subsection{The Helton-Howe formula for Toeplitz operators}
While the Helton-Howe formula is customarily discussed for $C^\infty$ symbols, we discuss it for symbols with 
least possible regularity. 
For this purpose, we first set up relevant function spaces. 
The contents of this subsection are mostly expository. 

For $f\in L^1(\T)$ and $n\in \Z$, we denote by $\hg(n)$ the Fourier coefficient 
$$\hf(n)=\frac{1}{2\pi}\int_\T f(t)e^{-int}dt.$$ 
Let $H^2$ be the Hardy space 
$$\{f\in L^2(\T);\;\forall n<0,\;\hf(n)=0\}.$$ 
We set $H^\infty=H^2\cap L^\infty(\T)$. 
Let $P_+$ be the projection from $L^2(\T)$ onto $H^2$, called the Riesz projections, and let $P_-=1-P_+$. 
For $f\in L^\infty(\T)$, we denote by $M_f\in B(L^2(\T))$ the multiplication operator $M_fh=fh$ for $h\in L^2(\T)$. 
The Toeplitz operator $T_f\in B(H^2)$ and the Hankel operator $H_f\in B(H^2,{H^2}^\perp)$ are defined by 
$T_f=P_+M_fP_+$ and $H_f=P_-M_fP_+$ respectively. 

For $s>0$, we denote by $W_2^{s}(\T)$ the Sobolev space 
$$W_{2}^s(\T)=\{f\in L_2(\T);\; \sum_{n\in \Z}(1+|n|^{2s})|\hf(n)|^2<\infty\},$$
and define the Sobolev norm by
$$\|f\|_{W_2^s}=\sqrt{\sum_{n\in \Z}(1+|n|^{2s})|\hf(n)|^2}.$$
The space $W_2^{1/2}(\T)$ is particularly important for us.  

It is straightforward to compute the Hilbert-Schmidt norm of $H_f$ as 
$$\|H_f\|_2^2=\sum_{n=0}^\infty\|H_fe_n\|_2^2=\sum_{n=1}^\infty n|\hf(-n)|^2,$$ 
where $e_n(t)=e^{int}$. 
Thus $H_f\in S_2(H^2,{H^2}^\perp)$ if and only if $P_-f\in W_2^{1/2}(\T)$. 
We have 
$$\|H_f\|_2^2+\|H_{\overline{f}}\|_2^2=\sum_{n\in \Z}|n||\hf(n)|^2
=\frac{1}{(2\pi)^2}\int_{\T^2}\frac{|f(s)-f(t)|^2}{|e^{is}-e^{it}|^2}dsdt,$$
where the second equality can be shown by the Fourier expansion. 
In particular,
\begin{equation}
\|f\|_{W_2^{1/2}}^2=\|f\|_2^2+\frac{1}{(2\pi)^2}\int_{\T^2}\frac{|f(s)-f(t)|^2}{|e^{is}-e^{it}|^2}dsdt.
\end{equation}

For $f,g\in W_2^{1/2}(\T)$, we have 
\begin{equation}
T_{fg}-T_fT_g=H_{\overline{f}}^*H_g\in S_1(H^2),
\end{equation}
and 
\begin{equation}
\|[T_f,T_g]\|_1\leq \|f\|_{W_2^{1/2}}\|g\|_{W_2^{1/2}}. 
\end{equation}
Indeed, 
\begin{align*}
\|[T_f,T_g]\|_1&=\|-H_{\overline{f}}^*H_g+H_{\overline{g}}^*H_f\|_1   \\
 &\leq \|H_{\overline{f}}\|_2 \|H_g\|_2+\|H_{\overline{g}}\|_2\|H_f\|_2\\
 &\leq \sqrt{(\|H_f\|_2^2+\|H_{\overline{f}}\|_2^2)(\|H_g\|_2^2+\|H_{\overline{g}}\|_2^2)}\\
 &\leq \|f\|_{W_2^{1/2}}\|g\|_{W_2^{1/2}}. 
\end{align*}

For $f,g\in W_2^{1/2}(\T)$, we define a skew-symmetric form $\omega$ by 
$$\omega(f,g)=\sum_{n\in \Z}n\hf(-n)\hg(n),$$
which satisfies $|\omega(f,g)|\leq \|f\|_{W_2^{1/2}}\|g\|_{W_2^{1/2}}$. 
Then the Helton-Howe formula takes the following form: 
$$\Tr([T_f,T_g])=\omega(f,g).$$ 
Indeed, the formula can be directly verified for trigonometric polynomials, and extends to $W_2^{1/2}(\T)$ by continuity. 

Let $W_1^1(\T)$ be the set of absolutely continuous functions on $\T$. 
Throughout the paper, we mainly work on the Toeplitz operators with symbols in $W_2^{1/2}(\T)\cap W_1^1(\T)$ and 
sometimes in $W_2^{1/2}(\T)\cap L^\infty(\T)$ or $W_2^{1/2}(\T)\cap H^\infty$. 
For the former, the Helton-Howe formula takes the following form:   

\begin{lemma} For $f,g\in W_2^{1/2}(\T)\cap W_1^1(\T)$, we have 
\begin{equation}\label{HH1}
\Tr([T_f,T_g])=\frac{1}{2\pi i}\int_{\T}f(t)g'(t)dt=\frac{-1}{2\pi i}\int_{\T}f'(t)g(t)dt.
\end{equation}
\end{lemma} 

\begin{proof} Let $F_n(t)$ be the Fej\'er kernel. 
Since $g'\in L^1(\T)$, the sequence $\{F_n*g'\}_n$ converges to $g'$ in $L^1(\T)$, and   
$$\frac{1}{2\pi i}\int_{\T}f(t)g'(t)dt=\lim_{n\to \infty}\frac{1}{2\pi i}\int_{\T}f(t)(F_n*g')(t)dt,$$
as $f\in L^\infty(\T)$. 
Since $g$ is absolutely continuous, we have $\hat{g'}(n)=in\hg(n)$, and 
$$\frac{1}{2\pi i}\int_{\T}f(t)g'(t)dt=\lim_{n\to \infty}\sum_{k=-n}^n\hf(-k)(1-\frac{|k|}{n})k\hg(k) 
 =\omega(f,g).$$
\end{proof}

\begin{example}
Note that we have the inclusion relations
$$C^1(\T)\subset W_2^1(\T)\subset W_2^{1/2}(\T)\cap W_1^1(\T).$$ 
The family of functions $f_\alpha(t)=(e^{it}+1)^\alpha$, $\alpha> 0$, illustrates the hierarchy of these functions spaces well.  
We have $f_\alpha\in W_2^{1/2}(\T)\cap W_1^1(\T)$ for all $\alpha>0$ (see Example \ref{ExAnBe}).  
We have $f_\alpha\in W_2^1(\T)$ if and only if $\alpha>1/2$, and $f\in C^1(\T)$ if and only if $\alpha\geq 1$.  
\end{example}

As Eq.(\ref{HH1}) does not make sense for symbols in $W_2^{1/2}(\T)\cap L^\infty(\T)$, we use harmonic extension 
to get an integral expression.  
For $f\in W_2^{1/2}(\T)\cap L^\infty(\T)$, we denote by $\tf$ the harmonic extension of $f$ to $\D$, that is, 
$\tf(re^{it})=P_r*f(t)$, where $P_r$ is the Poisson kernel. 
For later use, we present a more general trace formula than the Helton-Howe formula, which is of interest in its own right. 

\begin{lemma}\label{cyclic} Let $f,g\in W_2^{1/2}(\T)\cap L^\infty(\T)$, and let $h\in L^\infty(\T)$. 
Then we have    
\begin{equation}\label{HH2}
\Tr(T_h[T_f,T_g])=\frac{1}{2\pi i}\int_\D \tih d\tf\wedge d \tg.
\end{equation}
\end{lemma}

\begin{proof} Note that we have 
$$|\Tr(T_h[T_f,T_g])|\leq \|T_h\|\|[T_f,T_g]\|_1\leq \|h\|_\infty\|f\|_{W_2^{1/2}}\|g\|_{W_2^{1/2}}.$$

For $z\in \D$, we set  
$$F_+(z)=\sum_{n=0}^\infty \hf(n)z^n,\quad F_-(z)=\sum_{n=1}^\infty\overline{\hf(-n)}z^n.$$
$$G_+(z)=\sum_{n=0}^\infty \hg(n)z^n,\quad G_-(z)=\sum_{n=1}^\infty\overline{\hg(-n)}z^n.$$
Then $\tf(z)=F_+(z)+\overline{F_-(z)}$, $\tg(z)=G_+(z)+\overline{G_-(z)}$, and  
$$d\tf\wedge d\tg=dF_+\wedge d\overline{G_-}+d\overline{F_-}\wedge dG_-=(F_+'\overline{G_-'}-\overline{F_-'}G_+')
dz\wedge d\overline{z}.$$
Note that we have 
$$\frac{1}{\pi}\int_\D|F_+'(z)|^2dA(z)=\sum_{n=0}^\infty|n||\hf(n)|^2<\infty,$$
and a similar statement holds for $F_-$, $G_+$, and $G_-$. 
Thus we get 
\begin{align*}
\left|\frac{1}{2\pi i}\int_\D \tih d\tf\wedge d \tg\right| 
&=\left|\frac{1}{\pi}\int_\D \tih(z)(\overline{F_-'(z)}G_+'(z)-F_+'(z)\overline{G_-'(z)})dA(z)\right|\\
 &\leq \|h\|_\infty\|f\|_{W_2^{1/2}}\|g\|_{W_2^{1/2}}.
\end{align*}

The above two estimates show that it suffices to prove the statement for trigonometric polynomials $f$ and $g$, 
and furthermore it suffices to show 
$$\Tr(T_h[T_{e_{-m}},T_{e_{n}}])=\frac{mn}{\pi}\int_\D \tih(z)z^{n-1}\overline{z}^{m-1}dA(z),$$
for $m,n>0$.  
The left-hand side is 
$$\sum_{j=\max\{0,m-n\}}^{n-1}\inpr{he_{j+n-m}}{e_j}=\min\{m,n\}\hat{h}(m-n).$$
The right-hand side is 
\begin{align*}
\frac{mn}{\pi}\int_0^1\int_0^{2\pi}\tih(re^{i\theta})r^{m+n-1}e^{i(n-m)\theta}d\theta dr 
&=2mn\hat{h}(m-n)\int_0^1 r^{|n-m|+m+n-1}dr \\
 &=\frac{mn}{\max\{m,n\}}\hat{h}(m-n), 
\end{align*}
which finishes the proof. 
\end{proof}

\begin{remark} When $f$ and $g$ have sufficient regularity (and $h=1$), the two forms of the Helton-Howe formula above 
are, of course, directly connected by the Stokes theorem as suggested in \cite{HH1973}. 
\end{remark}

The space $W_2^{1/2}(\T)\cap L^\infty(\T)$ is known as the Krein algebra (see \cite[Chapter I, Section 8.11]{K2004}).  
As our argument heavily relies on the fact that it is an algebra, we include an elementary proof of it here. 

\begin{lemma} For $f,g\in W_{2}^{1/2}(\T)\cap L^\infty(\T)$, we have 
$$\|fg\|_{W_2^{1/2}}^2\leq 2\|f\|_{W_2^{1/2}}^2\|g\|_\infty^2+2\|f\|_\infty^2\|g\|_{W_2^{1/2}}^2. $$
In particular, the spaces $W_{2}^{1/2}(\T)\cap L^\infty(\T)$ and $W_2^{1/2}(\T)\cap W_1^1(\T)$ are algebras. 
\end{lemma}

\begin{proof} Let $f,g\in W_2^{1/2}(\T)\cap L^\infty(\T)$.  
Since 
$$|f(s)g(s)-f(t)g(t)|\leq |f(s)-f(t)|\|g\|_{\infty}+\|f\|_\infty|g(s)-g(t)|,$$
we have
\begin{align*}
\lefteqn{\|fg\|_{W_2^{1/2}}^2=\|fg\|_2^2+\frac{1}{(2\pi)^2}\int_{\T^2}\frac{|f(s)g(s)-f(t)g(t)|^2}{|e^{is}-e^{it}|^2}dsdt } \\
 &\leq \|f\|_2^2\|g\|_\infty^2
 +\frac{1}{(2\pi)^2}\int_{\T^2}\frac{2|f(s)-f(t)|^2\|g\|_\infty^2+2\|f\|_\infty^2|g(s)-g(t)|^2}{|e^{is}-e^{it}|^2}dsdt \\
 &=\|f\|_2^2\|g\|_\infty^2+2\sum_{n\in \Z}|n||\hf(n)|^2\|g\|_\infty^2
 +2\|f\|_\infty^2\sum_{n\in \Z}|n||\hg(n)|^2 \\
&\leq 2\|f\|_{W_2^{1/2}}^2\|g\|_\infty^2+2\|f\|_\infty^2\|g\|_{W_2^{1/2}}^2.  
\end{align*}

Since $W_1^1(\T)$ and $W_{2}^{1/2}(\T)\cap L^\infty(\T)$ are algebras, so is $W_2^{1/2}(\T)\cap W_1^1(\T)$. 
\end{proof}

\subsection{Trace formulae for $\Tr(\varphi(T_f^*T_f)-\varphi(T_fT_f^*))$}
Although we do not use any general results in the Krein theory of spectral shift functions in this paper,  
it is convenient for us to use its setup. 
For the Krein theory, the reader is referred to \cite[Chapter 9]{S2012}.

We denote by $\cW_1(\R)$ the set of functions $f\in C^1(\R)$ such that $f'$ is the Fourier transform of a (finite) complex Borel measure on $\R$.  
Let $A$ and $B$ be self-adjoint operators acting on a Hilbert space $H$. 
For simplicity, we assume that $A$ and $B$ are bounded. 
If $A-B\in S_1(H)$, the Krein theory shows that there exists a unique function $\xi_{A,B}\in L^1(\R)$, 
called the spectral shift function, such that 
\begin{equation}\label{tc}
\varphi(A)-\varphi(B)\in S_1(H),
\end{equation}
for all $\varphi\in \cW_1(\R)$, and the following Krein formula holds: 
\begin{equation}\label{KF}
\Tr(\varphi(A)-\varphi(B))=\int_\R \varphi'(x)\xi_{A,B}(x)dx.
\end{equation}

For a finite closed interval $[a,b]$, we denote by $\cW_1[a,b]$ the set of $f\in C^1[a,b]$ that extends to a function in 
$\cW_1(\R)$. 
Note that if $f\in C^1[a,b]$ and $f'$ extends to a periodic function with an absolutely convergent Fourier series, 
then $f\in \cW_1[a,b]$. 
In particular, if $f'$ satisfies either of the following conditions, 
\begin{itemize}
\item [(1)] $f'$ is absolutely continuous and $f''\in L^2[a,b]$, i.e. $f\in W_2^2[a,b]$,
\item [(2)] $f'$ is $\alpha$-H\"older continuous with some exponent $\alpha>1/2$, 
\item [(3)] $f'$ is of bounded variation and $\alpha$-H\"older continuous with some exponent $\alpha>0$, 
\end{itemize}
then $f\in \cW_1[a,b]$ (see \cite[Chapter I, Section 6]{K2004}). 
For example, the function $x^q$ belongs to $\cW_1[0,a]$ if and only if $q\geq 1$.  

It is known that Eq.(\ref{tc}) and (\ref{KF}) hold for a broader class of functions than $\cW_1(\R)$ 
(see \cite[Section 4.2]{BY1993}), but the function space $\cW_1(\R)$ is enough for our purpose. 

Let $f\in W_2^{1/2}(\T)\cap W_1^1(\T)$. 
Then since $T_f$ is almost normal, the Krein theory is applicable to the pair $(T_f^*T_f,T_fT_f^*)$. 
Our task now is to give an integral expression in terms of $f$ and $\varphi$ of the trace 
$\Tr(\varphi(T_f^*T_f)-\varphi(T_fT_f^*))$ without using the Krein theory. 

\begin{lemma}\label{holomorphic} 
Let $f\in W_{2}^{1/2}(\T)\cap W_1^{1}(\T)$, and let $\varphi$ be a holomorphic function 
defined on $\{z\in \C;|z|<r\}$ with $\|f\|_\infty^2<r$. 
Then we have $\varphi(T_f^*T_f)-\varphi(T_fT_f^*)\in S(H^2)_1$, and 
$$\Tr(\varphi(T_f^*T_f)-\varphi(T_fT_f^*))=\frac{1}{2\pi i}\int_\T \Phi(|f(t)^2|)\overline{f(t)}f'(t)dt,$$
where 
$$\Phi(x)=\left\{
\begin{array}{ll}
\frac{\varphi(x)-\varphi(0)}{x} , &\quad x\neq 0  \\
\varphi'(0) , &\quad x=0
\end{array}
\right.
$$

\end{lemma}

\begin{proof}
With $a_n=\varphi^{(n)}(0)/n!$, we have 
$$\varphi(T_f^*T_f)-\varphi(T_fT_f^*)=\sum_{n=0}^\infty a_n\left( (T_f^*T_f)^n-(T_fT_f^*)^n\right),$$
where the convergence is in the operator norm. 
We first claim that it actually converges in the trace norm too. 
Indeed, since 
$$(T_f^*T_f)^n-(T_fT_f^*)^n=[(T_f^*T_f)^{n-1}T_f^*,T_f],$$ 
we get 
$$\|(T_f^*T_f)^n-(T_fT_f)^n\|_1\leq n\|[T_f^*,T_f]\|_1\|T_f\|^{2n-2}\leq n\|f\|_{W_2^{1/2}}^2\|f\|_\infty^{2n-2},$$
which shows the claim as the radius of convergence of the Taylor series for $\varphi$ is strictly larger than $\|f\|_\infty^2$. 
Thus $\varphi(T_f^*T_f)-\varphi(T_fT_f^*)\in S_1(H^2)$. 

Since $(T_f^*T_f)^{n-1}T_f^*-T_{|f|^{n-1}\overline{f}}\in S_1(H^2)$, the Helton-Howe formula implies 
$$\Tr([(T_f^*T_f)^{n-1}T_f^*,T_f])=\Tr([T_{|f|^{2(n-1)}\overline{f}},T_f])
=\frac{1}{2\pi i}\int_\T |f(t)|^{2(n-1)}\overline{f(t)}f'(t)dt,$$ 
and so 
$$\Tr(\varphi(T_f^*T_f)-\varphi(T_fT_f^*))=\frac{1}{2\pi i}\sum_{n=1}^\infty a_n \int_\T |f(t)|^{2(n-1)}\overline{f(t)}f'(t)dt.$$
Since 
$$\sum_{n=1}^\infty a_n|f(t)|^{2n-2}=\Phi(|f(t)|^2)$$
converges uniformly on $\T$, we get the statement.  
\end{proof}

In what follows, for simplicity we often write 
$$\Phi(|f(t)|^2)\overline{f(t)}f'(t)=\left(\varphi(|f(t)|^2)-\varphi(0)\right)\frac{f'(t)}{f(t)},$$
with convention that the right-hand side is 0 whenever $f(t)=0$.

\begin{theorem}\label{HHtoKF} 
Let $f\in W_{2}^{1/2}(\T)\cap W_1^{1}(\T)$. 
Then the following holds for every $\varphi\in \cW_1[0,\|f\|_\infty^2]$:
\begin{equation}\label{KFfT}
\Tr(\varphi(T_f^*T_f)-\varphi(T_fT_f^*))=\frac{1}{2\pi i}\int_\T \left(\varphi(|f(t)^2|)-\varphi(0)\right)\frac{f'(t)}{f(t)}dt. 
\end{equation}
\end{theorem}

\begin{proof} From the definition of $\cW_1[0,\|f\|_\infty]$, there exists a complex measure $\mu$ on $\R$ such that 
$$\varphi(x)=\varphi(0)+\int_\R \frac{e^{ixy}-1}{iy}d\mu(y),$$
and 
$$\varphi(T_f^*T_f)-\varphi(T_fT_f^*)=\int_\R \frac{1}{iy}(e^{iy T_f^*T_f}-e^{iy T_fT_f^*})d\mu(y).$$
Note that the right-hand side converges in the trace norm (see \cite[Lemma 9.26]{S2012}). 
Thus 
$$\Tr(\varphi(T_f^*T_f)-\varphi(T_fT_f^*))=\int_\R \frac{1}{iy}
\Tr(e^{iy T_f^*T_f}-e^{iy T_fT_f^*})d\mu(y).$$
Now the statement follows from the previous lemma and the Fubini theorem. 
\end{proof}

\begin{cor}
Let $f\in W_{2}^{1/2}(\T)\cap W_1^{1}(\T)$. 
Then the following holds for every $\varphi\in \cW_1[0,\|f\|_\infty^2]$:
\begin{equation}\label{KFD}
\Tr(\varphi(T_f^*T_f)-\varphi(T_fT_f^*))=\frac{1}{2\pi i}\int_\D \varphi'(|\tf|^2)d\overline{\tf}\wedge d\tf.
\end{equation}
\end{cor}

\begin{proof} If $f$ is smooth, the statement follows from the previous theorem together with the Stokes theorem. 
For general $f$ and $0<r<1$, let $f_r(t)=\tf(re^{it})$. 
Since $f$ is absolutely continuous, we have $(f_r)'=(f')_r$, and 
$$\lim_{r\to 1-0}\|(f_r)'-f'\|_1=0,\quad \lim_{r\to 1-0}\|f_r-f\|_\infty=0.$$
Thus \begin{align*}
\lefteqn{\Tr(\varphi(T_f^*T_f)-\varphi(T_fT_f^*))
=\lim_{r\to 1-0} \frac{1}{2\pi i} \int_\T \Phi(|f_r(t)|^2)\overline{f_r(t)}(f_r)'(t)dt} \\
 &=\lim_{r\to 1-0} \frac{1}{2\pi i}\int_\D \varphi'(|\tilde{f_r}|^2)d\overline{\tilde{f_r}}\wedge d\tilde{f_r}
 =\lim_{r\to 1-0} \frac{1}{2\pi i}\int_{|z|\leq r} \varphi'(|\tf|^2)d\overline{\tf}\wedge d\tf.
\end{align*}
Since $\varphi'(|\tf(z)|^2)$ is bounded and the form $d\overline{\tf}\wedge d\tf$ is integrable on $\D$, 
we get the statement. 
\end{proof}

\begin{remark}\label{ssf} 
We can describe the spectral shift function $\xi_f:=\xi_{T_f^*T_f,T_fT_f^*}$ for $f\in W_2^{1/2}(\T)\cap W_1^1(\T)$ as 
\begin{equation}\label{SSF}
\xi_f(x)=\frac{1}{2\pi i}\int_\T 1_{(0,|f(t)|^2]}(x)\frac{f'(t)}{f(t)}dt,
\end{equation}
where $1_X$ is the indicator function of $X\subset \R$. 
This can be shown by comparing the Fourier transform of the right-hand side with Lemma \ref{holomorphic}.
Alternatively, the method explained in Introduction shows 
$$\xi_f(x)=\frac{1}{2\pi}\int_\T\frac{1}{2\pi i}\int_\T\frac{f'(t)}{f(t)-\sqrt{x}e^{i\theta}}dtd\theta.$$
If the Fubini theorem is applicable to this iterated integral, we get the same formula.  
\end{remark}

Now we relax the regularity of the symbol $f$ while imposing analyticity. 
We say that a symbol $f\in W_2^{1/2}(\T)\cap L^\infty(\T)$ is analytic if $f\in H^\infty$.  
For analytic $f$, we have $[T_f^*,T_f]=H_{\overline{f}}^*H_{\overline{f}}\geq 0$. 

\begin{lemma}
Let $f\in W_{2}^{1/2}(\T)\cap H^\infty$, and let $F(z)=\tf(z)$ be the holomorphic extension of $f$ to $\D$. 
Let $\varphi$ be a holomorphic function defined on $\{z\in \C;|z|<r\}$ with $\|f\|_\infty^2<r$. 
Then we have 
$$\Tr(\varphi(T_f^*T_f)-\varphi(T_fT_f^*))=\frac{1}{\pi}\int_\D \varphi'(|F(z)|^2) |F'(z)|^2dA(z).$$
\end{lemma}

\begin{proof} We use the notation in the proof of Lemma \ref{holomorphic}. 
Lemma \ref{cyclic} implies 
\begin{align*}
\lefteqn{\Tr([T_{|f|^{2n-2}\overline{f}},T_f])=\Tr([T_{\overline{f}^n f^{n-1}},T_f])=
\Tr([T_{\overline{f}^n},T_f]T_{f^{n-1}})} \\
 &=\frac{1}{2\pi i}\int_\D F^{n-1}d\overline{F^n}\wedge dF
 =\frac{1}{\pi}\int_\D n|F(z)^2|^{n-1}|F'(z)|^2dA(z), 
\end{align*}
and 
\begin{align*}
\Tr(\varphi(T_f^*T_f)-\varphi(T_fT_f^*)) &=\frac{1}{\pi}\sum_{n=1}^\infty a_n\int_\D n|F(z)^2|^{n-1}|F'(z)|^2dA(z) \\
 &=\frac{1}{\pi}\int_\D \varphi'(|F(z)|^2) |F'(z)|^2dA(z).
\end{align*}
\end{proof}

The above lemma together with a similar argument as in the proof of Theorem \ref{HHtoKF} implies 
the following theorem.  

\begin{theorem}\label{AHHtoKF}
Let $f\in W_{2}^{1/2}(\T)\cap H^\infty$, and let $F(z)=\tf(z)$ be the holomorphic extension of $f$ to $\D$. 
Then for every $\varphi\in \cW_1[0,\|f\|_\infty^2]$, we have $\varphi(T_f^*T_f)-\varphi(T_fT_f^*)\in S_1(H^2)$ and
\begin{equation}\label{AKF}
\Tr(\varphi(T_f^*T_f)-\varphi(T_fT_f^*))=\frac{1}{\pi}\int_\D \varphi'(|F(z)|^2)|F'(z)|^2dA(z). 
\end{equation}
\end{theorem}

\begin{remark} Eq.(\ref{AKF}) gives a geometric interpretation of the spectral shift function $\xi_f$ for analytic $f$. 
Note that $|F'(z)|^2$ is the Jacobian of the map $F:\D\to \C$. 
The Sard theorem shows that the critical values  
$$C(F)=F(\{z\in \D;\;F'(z)=0 \}),$$
has Lebesgue measure 0. 
We define the multiplicity function $m(w)=\# F^{-1}(w)$ for $w\in F(\D)\setminus C(F)$.  
Then Eq.(\ref{AKF}) shows 
$$\int_{0}^{\|f\|_\infty^2}\varphi'(x)\xi_f(x)dx=\frac{1}{\pi}\int_{F(\D)}\varphi'(|w|^2)m(w)dA(w).$$
As $\varphi'$ can be any function in $C^1[0,\|f\|_\infty^2]$, this means that $\xi_f$ is the density function, with respect to 
the Lebesgue measure on $[0,\|f\|_\infty^2]$, of the variable $|w|^2$ whose distribution is given by the measure  
$$\frac{m(w)dA(w)}{\pi}$$
on $F(\D)$. 
\end{remark}

\begin{remark} It is shown in \cite[Theorem 1.2]{TWZ2023} that the Helton-Howe formula holds for Toeplitz operators 
$T^{(t)}_f$ on the weighted Bergman spaces $L^2_{a,t}(\D)$ with symbols in $C^2(\overline{\D})$.  
The condition $T^{(t)}_{fg}-T^{(t)}_fT^{(t)}_g\in S_1(L^2_{a,t}(\D))$ holds too (\cite[Theorem 6.3,(1)]{TWZ2023}, 
\cite[Theorem 8.36]{Z2007}).  
As the Helton-Howe formula itself remains the same, 
$$\Tr([T^{(t)}_f,T^{(t)}_g])=\frac{1}{2\pi i}\int_\D df\wedge dg=\frac{1}{2\pi i}\int_{\partial \D} fdg,$$
the formulae Eq.(\ref{KFfT}), (\ref{KFD}), (\ref{SSF}), (\ref{AKF}) still hold in the case of the weighted Bergman spaces too.  
\end{remark}
\section{The Witten index formula}
Assume that $f\in C^1(\T)$ has no zeros. 
Since $T_{\frac{1}{f}}$ is the inverse of $T_f$ modulo $S_1(H^2)$, the Fredholm index of $T_f$ is given by 
$\Tr([T_f,T_{\frac{1}{f}}])$, 
and the Helton-Howe formula shows 
$$\ind T_f=\Tr([T_f,T_{\frac{1}{f}}])=\frac{-1}{2\pi i}\int_\T \frac{f'(t)}{f(t)}dt.$$ 
The purpose of this section is to generalize this formula to the Witten index 
as an application of Lemma \ref{holomorphic}. 

\begin{lemma} 
Let $a>0$, and let $f$ be an absolutely continuous function on $[-a,a]$ having the only zero at $x=0$. 
We assume:
\begin{itemize} 
\item[$(1)$] $f'(x)$ exists for every $x\in [-a,a]\setminus \{0\}$.  
\item[$(2)$] There exists $g\in C^1[-a,a]$ satisfying 
$$\frac{f'(x)}{f(x)}=\frac{g(x)}{x},\quad \forall x\in [-a,a]\setminus \{0\}.$$
\item[$(3)$] There exist $\beta>0$ and $h\in C^1[-a,a]$ such that $|f(x)|^2=|x|^\beta h(x)$ and 
$h(x)>0$ for all $x\in [-a,a]$.  
\end{itemize} 
Then for every $\varepsilon\in (0,a]$, the limit 
$$\lim_{s\to +\infty}\int_{-\varepsilon}^\varepsilon(1-e^{-s|f(x)|^2})\frac{f'(x)}{f(x)}dx$$
exists and it is the order of $O(\varepsilon)$ as $\varepsilon$ tends to $+0$. 
\end{lemma}

\begin{proof} We have \begin{align*}
\lefteqn{\int_{-\varepsilon}^\varepsilon(1-e^{-s|f(x)|^2})\frac{f'(x)}{f(x)}dx} \\
 &=g(0)\int_{-\varepsilon}^\varepsilon(1-e^{-s|x|^\beta h(x)})\frac{1}{x}dx
 +\int_{-\varepsilon}^\varepsilon(1-e^{-s|x|^\beta h(x)})\frac{g(x)-g(0)}{x}dx.
\end{align*} 
For the first term, we have 
\begin{align*}
\lefteqn{\int_{-\varepsilon}^\varepsilon(1-e^{-s|x|^\beta h(x)})\frac{1}{x}dx
=\int_0^\varepsilon (e^{-sx^\beta g(-x)}-e^{-sx^\beta g(x)})\frac{1}{x}dx } \\
 &=-\int_0^\varepsilon\int_{-1}^1\frac{\partial}{\partial r}e^{-x^\beta g(rx)}dr\frac{1}{x}dx
 =s\int_0^\varepsilon \int_{-1}^1x^\beta h'(rx)e^{-sx^\beta h(rx)}drdx \\
 &=s^{-1/\beta}\int_0^{\varepsilon s^{1/\beta}}\int_{-1}^1y^\beta h'(rs^{-1/\beta}y)e^{-y^\beta h(rs^{-1/\beta}y)}drdy. 
\end{align*}
Let $M=\|h'\|_\infty$ and $m=\min\{h(x);\;x\in [a,-a]\}$. 
Then 
$$|\int_{-\varepsilon}^\varepsilon(1-e^{-s|x|^\beta h(x)})\frac{1}{x}dx|\leq 
s^{-1/\beta}2M\int_0^{\infty}y^\beta e^{-m y^\beta}dy.$$
Thus 
$$\lim_{s\to+\infty}\int_{-\varepsilon}^\varepsilon(1-e^{-s|f(x)|^2})\frac{f'(x)}{f(x)}dx 
=\int_{-\varepsilon}^\varepsilon\frac{g(x)-g(0)}{x}dx=O(\varepsilon)$$
as $\varepsilon\to +0$. 
\end{proof}

Now we show a generalized winding number formula for the Witten index of a Toeplitz operator. 

\begin{theorem}\label{Witten} 
Let $f\in W_2^{1/2}(\T)\cap W_1^1(\T)$, and assume that $f$ has only finitely many zeros 
$t_1,t_2,\cdots, t_m$. 
We further assume that there exists $\delta>0$ satisfying the following for each $1\leq j\leq m$: 
\begin{itemize} 
\item[$(1)$] $f'(t)$ exists for every $t\in [t_j-\delta,t_j+\delta]\setminus \{t_j\}$.  
\item[$(2)$] There exists $g_j\in C^1[t_j-\delta,t_j+\delta]$ satisfying 
$$\frac{f'(t)}{f(t)}=\frac{g_j(t)}{t-t_j},\quad \forall t\in [t_j-\delta,t_j+\delta]\setminus \{t_j\}.$$
\item[$(3)$] There exist $\beta_j>0$ and $h_j\in C^1[t_j-\delta,t_j+\delta]$ such that $|f(t)|^2=|t-t_j|^{\beta_j} h_j(t)$, and 
$h_j(t)>0$ for all $t\in [t_j-\delta,t_j+\delta]$.  
\end{itemize} 
Then the Witten index for $T_f$ exists and it is given by the following principal value integral:
\begin{equation}
\ind_W T_f=\frac{-1}{2\pi i}\:\mathrm{p.v.}\int_{\T}\frac{f'(t)}{f(t)}dt.
\end{equation}
\end{theorem}

\begin{proof} Lemma \ref{holomorphic} implies 
$$\Tr(e^{-sT_f^*T_f}-e^{-sT_fT_f^*})=\frac{-1}{2\pi i}\int_\T (1-e^{-s|f(t)|^2})\frac{f'(t)}{f(t)}dt.$$
For $0<\varepsilon\leq \delta$, we set 
$$I_\varepsilon=\bigcup_{j=1}^m (t_j-\varepsilon,t_j+\varepsilon).$$ 
Then since $f'(t)/f(t)$ is integrable on $\T\setminus I_\varepsilon$, the Lebesgue theorem implies 
$$\lim_{s\to+\infty}\int_{\T\setminus I_\varepsilon}(1-e^{-s|f(t)|^2})\frac{f'(t)}{f(t)}dt=
\int_{\T\setminus I_\varepsilon}\frac{f'(t)}{f(t)}dt.$$
The above lemma implies that the following double limit exists:  
$$\lim_{\varepsilon \to +0}\lim_{s\to+\infty}\int_{I_\varepsilon}(1-e^{-s|f(t)|^2})\frac{f'(t)}{f(t)}dt=0.$$
Thus we get 
$$\lim_{s\to+\infty}\int_{\T}(1-e^{-s|f(t)|^2})\frac{f'(t)}{f(t)}dt
=\lim_{\varepsilon\to+0}\int_{\T\setminus I_\varepsilon}\frac{f'(t)}{f(t)}dt=\mathrm{p.v.}\int_{\T}\frac{f'(t)}{f(t)}dt.$$
\end{proof}

\begin{example} Let $Q(z)$ be a rational function without poles on the unit circle, and let $f(t)=Q(e^{it})$. 
We can express $Q(z)$ as 
$$Q(z)=c\frac{\prod_{k=1}^N(z-a_k)^{n_k}}{\prod_{j=1}^M (z-b_j)^{m_j}}.$$
Then we have 
\begin{align*}
\ind_WT_f&=\frac{1}{2\pi i}\:\mathrm{p.v.}\int_{|z|=1}\left(\sum_{j=1}^M\frac{m_j}{z-b_j}-\sum_{k=1}^N\frac{n_k}{z-a_k}\right)dz \\
 &=\sum_{|b_j|<1}m_j-\sum_{|a_k|<1}n_k-\frac{1}{2}\sum_{|a_k|=1}n_k.
\end{align*}
This is the reason why a half-integer appears as $\ind_W T_f$ in \cite{MSTTW2023}. 
\end{example}

\begin{example}\label{anyv}
 Let $f(t)=e^{int}(1+e^{it})^\alpha$ with $\alpha>0$ and $n\in \Z$. 
Then
$$\ind_WT_f=\frac{-1}{2\pi i}\:\mathrm{p.v.}\int_{|z|=1}\left(\frac{n}{z}+\frac{\alpha}{z+1}\right)dz=-n-\frac{\alpha}{2}.$$
This shows that the Witten index can take any real numbers. 
\end{example}

\begin{problem}
Is it possible to realize a value other than the half-integers as $\ind_W T_f$ with $f\in C^\infty(\T)$ ?
\end{problem}

\section{Extension of the trace formulae to operator monotone functions}\label{om}
Let $f\in W_2^{1/2}(\T)\cap W_1^1(\T)$. 
As the function $x^{p/2}$ belongs to $\cW_1[0,\|f\|_\infty^2]$ for $p\geq 2$, Theorem \ref{HHtoKF} implies  
\begin{equation}\label{ptr}
|T_f|^{p}-|T_f^*|^p\in S_1(H^2),
\end{equation}
\begin{equation}\label{pKF}
\Tr(|T_f|^p-|T_f^*|^p)=\frac{1}{2\pi i}\int_\T |f(t)|^p\frac{f'(t)}{f(t)}dt,
\end{equation} 
for $p\geq 2$. 
Although the Krein theory does not apply to the case of $0<p<2$, we still have a chance to get the same result if the symbol $f$ 
has better regularity. 
In fact, if the Hankel operators $H_f$ and $H_{\overline{f}}$ belong to $S_p(H^2,{H^2}^\perp)$, 
we will see that Eq.(\ref{ptr}) holds, and so does Eq.(\ref{pKF}) as far as the right-hand side makes sense. 
Moreover, when the symbol $f$ is analytic, we will show a much stronger result due to the hyponormality of $T_f$, i.e.  
$T_f^*T_f-T_fT_f^*\geq 0$.  

We first recall Peller's famous criterion for the Hankel operators to belong to $S_p(H^2,{H^2}^\perp)$ 
(see \cite[Chapter 6]{P2003}, \cite[Theorem 10.21]{Z2007} for the proof). 
Note that $H_f=H_{P_-f}$, and $\overline{P_-f}$ is analytic. 

\begin{theorem}\label{Peller} 
Let $0<p$, and let $f$ be an analytic symbol. 
Then $H_{\overline{f}}\in S_p(H^2,{H^2}^\perp)$ if and only if $f$ belongs to the Besov space $B_p^{1/p}(\T)$. 
\end{theorem}

The reader is referred to \cite[Appendix 2]{P2003} and \cite{P2024} for the basic properties of the Besov spaces. 
The space $B_{p}^{1/p}(\T)$ is closed under the Riesz projection $P_+$ for every $p>0$. 
It is known that we have $B_{2}^{1/2}(\T)=W_2^{1/2}(\T)$ and $B_1^1(\T)\subset W_1^1(\T)$ (see, for example, 
\cite[Section 1,(P9)]{GGJ2011}). 
As we have $B_1^1(\T) \subset B_2^{1/2}(\T)$ (corresponding to the inclusion $S_1(H^2,{H^2}^\perp)\subset S_2(H^2,{H^2}^\perp)$ 
through Peller's theorem), every symbol in $B_1^1(\T)$ belongs to our working space $W_2^{1/2}(\T)\cap W_1^1(\T)$. 

We denote by $B_p$ the set of holomorphic functions on $\D$ whose boundary value 
functions belong to $B_p^{1/p}(\T)$, and call it the analytic Besov space.  
By slightly abusing notation, we write $f\in B_p$ if $\tf\in B_p$.  
Since Peller's theorem only requires $B_p$, we recall the following useful criterion for $B_p$ instead of giving 
the definition of $B_p^{1/p}(\T)$ 
(in fact, Zhu's book \cite[Chapter 5.3]{Z2007} adopts it as the definition of $B_p$). 

\begin{lemma}\label{B_p}  
For a holomorphic function $F(z)$ on $\D$ and $p>0$, the following conditions are equivalent: 
\begin{itemize} 
 \item [$(1)$] $F\in B_p$. 
 \item [$(2)$] There exists $n\in \N$ satisfying $pn>1$ and 
 $$\int_\D\left|(1-|z|^2)^n F^{(n)}(z)\right|^p\frac{1}{(1-|z|^2)^2}dA(z)<\infty.$$ 
 \item [$(3)$] For every $n\in \N$ satisfying $pn>1$, 
 $$\int_\D\left|(1-|z|^2)^n F^{(n)}(z)\right|^p\frac{1}{(1-|z|^2)^2}dA(z)<\infty.$$ 
\end{itemize} 
\end{lemma}

In particular, a holomorphic function $F$ on $\D$ belongs to $B_1$ if and only if 
$$\int_\D|F''(z)|dA(z)<\infty,$$
and it belongs to $B_2$ if and only if 
$$\int_\D|F'(z)|^2 dA(z)<\infty.$$

\begin{example}\label{ExAnBe} 
Let $\alpha>0$ and let $\psi(z)=\frac{z}{\log(1+z)}$. 
Note that $\psi$ has no zero in $\D$, and extends to a continuous function on $\overline{\D}$ 
with the only zero at $z=-1$. 
Thus $\log \psi(z)$ is well-defined as a continuous function on $\overline{\D}\setminus \{-1\}$, 
where we take the branch satisfying $\log \psi(1)>0$.  
As $\lim_{z\to -1}\log\psi(z)=-\infty$, we can choose sufficiently large $C>0$ so that 
$\frac{1}{C-\log \psi(z)}$ continuously extends to $\overline{\D}$. 
We fix such $C$. 
Direct computation using Lemma \ref{B_p} shows: 
\begin{itemize}
\item [(1)] $(1+z)^\alpha$ belongs to $B_p$ for all $p>0$. 
\item [(2)] $\psi(z)^\alpha$ belongs to $B_p$ if and only if $p>\frac{1}{1+\alpha}$. 
\item [(3)] $\frac{1}{C-\log\psi(z)}$ belongs to $B_p$ if and only if $p\geq 1$. 
\end{itemize}
\end{example}

We fix $p>0$ for now and set $q=p/2$ to avoid possible confusion. 
Since the function $x^{p/2}=x^q$ in the region $0<p<2$ is operator monotone, 
we can make use of the sophisticated theory of operator monotone functions on $[0,\infty)$ to extend Eq.(\ref{ptr}) 
and Eq.(\ref{pKF}). 
Recall that a function $\varphi:[0,\infty)\to \R$ is said to be operator monotone if $A\geq B\geq 0$ implies 
$\varphi(A)\geq \varphi(B)$ for any positive operators $A,B\in B(H)$ on a Hilbert space $H$. 
For $\lambda> 0$, we set $\varphi_{\lambda}(x)=\frac{x}{\lambda+x}$, which 
is a typical example of operator monotone functions. 
It is well-known (see for example \cite[V53]{B1997}) that for an operator monotone function $\varphi$ on $[0,\infty)$, 
there exist a unique constant $a_\varphi\geq 0$ and a positive Borel measure $\mu_\varphi$ on $[0,\infty)$ satisfying 
$$\int_0^\infty \frac{\lambda}{1+\lambda}d\mu_\varphi(\lambda)<\infty,$$
such that 
\begin{equation}\label{ieomf}
\varphi(x)=\varphi(0)+a_\varphi x+\int_0^\infty \varphi_\lambda(x)\lambda d\mu_\varphi(\lambda).
\end{equation}
For $\varphi(x)=x^{q}$, we have $a_\varphi=0$, and $\mu_\varphi(d\lambda)=\frac{\sin{q\pi}}{\pi}\lambda^{q-2}d\lambda$. 

If $A-B$ is a compact operator, we can see that $\varphi(A)-\varphi(B)$ is a compact operator too. 
Indeed, the integral expression shows
$$\varphi(A)-\varphi(B)=a_\varphi(A-B)+\int_0^\infty(\varphi_\lambda(A)-\varphi_\lambda(B))\lambda d\mu_\varphi(\lambda),$$
where the convergence is, a priori, in the strong operator topology. 
As the integral is equal to 
$$\lim_{n\to\infty}\int_{\frac{1}{n}}^n(\frac{1}{\lambda+B}-\frac{1}{\lambda +A})\lambda^2 d\mu_\varphi(\lambda)=
\lim_{n\to\infty}\int_{\frac{1}{n}}^n 
\frac{1}{\lambda+B}(A-B)\frac{1}{\lambda +A}\lambda^2 d\mu_\varphi(\lambda)
,$$
the convergence is the norm topology, and we see that $\varphi(A)-\varphi(B)$ is a compact operator.  

If moreover $A-B\in S_1(H)$, the map 
$$(0,\infty)\ni \lambda\mapsto \frac{1}{\lambda+B}-\frac{1}{\lambda +A}=\frac{1}{\lambda+B}(A-B)\frac{1}{\lambda +A}$$
is continuous in the trace norm. 
Note that if moreover $A\geq B$, the quantity 
$$\Tr(\varphi(A)-\varphi(B))\in [0,\infty]$$
makes sense regardless of whether $\varphi(A)-\varphi(B)$ belongs to $S_1(H)$ or not. 
Thus the lower semicontinuity of the trace implies 
\begin{equation}\label{trint}
\Tr(\varphi(A)-\varphi(B))=a_\varphi\Tr(A-B)+\int_0^\infty\Tr\left(\frac{1}{\lambda+B}-\frac{1}{\lambda+A}\right)\lambda^2 
d\mu_\varphi(\lambda).
\end{equation}

We recall a majorization result involving an operator monotone function due to Ando \cite{A2988}
and Kosaki \cite[Appendix]{HN1989}. 

\begin{theorem} Let $A, B\in B(H)$ be positive operators on a Hilbert space $H$, and assume that $A-B$ is 
a positive compact operator. 
Let $\varphi$ be an operator monotone function on $[0,\infty)$ satisfying $\varphi(0)=0$. 
Then for every $n\in \N$, the majorization inequality
$$\sum_{j=1}^n s_k(\varphi(A)-\varphi(B))\leq \sum_{k=1}^n s_k(\varphi(A-B)).$$
holds. 
In particular, if $\varphi(A-B)$ belongs to $S_1(H)$, so does $\varphi(A)-\varphi(B)$, and  
$$\Tr(\varphi(A)-\varphi(B))\leq \Tr(\varphi(A-B)).$$
\end{theorem}
 
As we use only the following special case, we state it separately and give an elementary proof. 

\begin{cor}\label{qtrace} 
Let $0<q<1$, and let $A,B\in B(H)$ be positive operators. 
We further assume $A\geq B$ and $A-B\in S_q(H)$. 
Then $A^q-B^q\in S_1(H)$ and 
$$\Tr(A^q-B^q)\leq \Tr((A-B)^q).$$ 
\end{cor} 

\begin{proof} Let $D=A-B$. 
Thanks to Eq.(\ref{trint}), it suffices to show
$$\frac{\sin(q\pi)}{\pi}\int_0^\infty \Tr\left(\frac{1}{\lambda+B}-\frac{1}{\lambda+B+D}\right)\lambda^qd\lambda\leq \Tr(D^q).$$
The resolvent identity shows 
\begin{align*}
\lefteqn{\frac{1}{\lambda+B}-\frac{1}{\lambda+B+D}=(\lambda+B)^{-1}D(\lambda+B+D)^{-1}} \\
 &=(\lambda+B)^{-1}D\left(1+(\lambda+B)^{-1}D\right)^{-1}(\lambda+B)^{-1}\\
  &=(\lambda+B)^{-1}D^{1/2}\left(1+D^{1/2}(\lambda+B)^{-1}D^{1/2}\right)^{-1}D^{1/2}(\lambda+B)^{-1},\\
\end{align*}
and 
\begin{align*}
\lefteqn{\Tr\left(\frac{1}{\lambda+B}-\frac{1}{\lambda+B+D}\right)} \\
 &=\Tr\left( (\lambda+B)^{-1}D^{1/2}\left(1+D^{1/2}(\lambda+B)^{-1}D^{1/2}\right)^{-1}D^{1/2}(\lambda+B)^{-1}\right)\\
 &=\Tr\left(\left(1+D^{1/2}(\lambda+B)^{-1}D^{1/2}\right)^{-1}D^{1/2}(\lambda+B)^{-2}D^{1/2}\right)\\
 &=-\frac{d}{d\lambda}\Tr\left(\log \left(1+D^{1/2}(\lambda+B)^{-1}D^{1/2}\right)\right), 
\end{align*}
where we used \cite[Lemma 9.16]{S2012}. 
Thus letting 
$$h(\lambda)=\Tr\left(\log \left(1+D^{1/2}(\lambda+B)^{-1}D^{1/2}\right)\right),$$
we get
$$\frac{\sin(q\pi)}{\pi}\int_0^\infty \Tr\left(\frac{1}{\lambda+B}-\frac{1}{\lambda+B+D}\right)\lambda^qd\lambda=
-\frac{\sin(q\pi)}{\pi}\int_0^\infty h'(\lambda)\lambda^qd\lambda.$$

We claim that integration by part implies  
$$-\frac{\sin(q\pi)}{\pi}\int_0^\infty h'(\lambda)\lambda^qd\lambda
=\frac{q\sin(q\pi)}{\pi}\int_{0}^\infty h(\lambda)\lambda^{q-1}dq.$$
To verify the claim, it suffices to show the convergence of the right-hand side, and 
$$\lim_{\lambda\to +0}h(\lambda)\lambda^q=0,$$
as we have $h(\lambda)=O(1/\lambda)$ for large $\lambda$. 
In fact, since $h(\lambda)$ is non-negative and monotone decreasing, the latter follows from the former. 
Since 
$$0\leq h(\lambda)\leq \Tr(\log(1+\frac{1}{\lambda}D))=\sum_{n=1}^\infty \log(1+\frac{s_n(D)}{\lambda}),$$
and 
$$\sum_{n=1}^\infty \int_0^\infty \log(1+\frac{s_n(D)}{\lambda})\lambda^{q-1}d\lambda=
\sum_{n=1}^\infty s_n(D)^q\int_0^\infty \log(1+\frac{1}{t})t^{q-1}dt,$$
we get 
$$\frac{q\sin(q\pi)}{\pi}\int_{0}^\infty h(\lambda)\lambda^{q-1}dq
\leq \Tr(D^q)\frac{q\sin(q\pi)}{\pi}\int_0^\infty \log(1+\frac{1}{t})t^{q-1}dt<\infty,$$
and the claim is shown. 

Integration by part again implies  
$$\frac{q\sin(q\pi)}{\pi}\int_0^\infty \log(1+\frac{1}{t})t^{q-1}dt=\frac{\sin(q\pi)}{\pi}\int_0^\infty \frac{t^{q-1}}{1+t}dt=1,$$
and the proof is finished. 
\end{proof}

\begin{remark} The function $h(\lambda)$ in the above proof is nothing but 
$$\int_\R \frac{\xi_{A,B}(x)}{\lambda+x}dx.$$
\end{remark}

\begin{lemma} Let $0<q<1$, and let $A,B\in B(H)$ be positive operators satisfying $A-B\in S_q(H)$. 
Then $A^q-B^q\in S_1(H)$ and 
$$\Tr(A^q-B^q)=\frac{\sin(q\pi)}{\pi}\int_0^\infty \Tr\left(\frac{1}{\lambda+B}-\frac{1}{\lambda+A}\right)\lambda^qd\lambda,$$
where the integral in the right-hand side converges absolutely. 
\end{lemma}

\begin{proof} Let $D=A-B$, and let $D_+$ and $D_-$ be the positive part and negative part of $D$ respectively. 
Then $D_+,D_-\in S_q(H)$, and $A+D_-=B+D_+$. 
Now the statement follows from Corollary \ref{qtrace} and 
$$A^q-B^q=(B+D_+)^q-B^q-\left( (A+D_-)^q-A^q \right).$$ 
\end{proof}

\begin{theorem}\label{pthm} 
Let $f\in W_2^{1/2}(\T)\cap W_1^1(\T)$, and let $0<p<2$. 
If $f\in B_p^{1/p}(\T)$, Eq.(\ref{ptr}) holds. 
If moreover 
$$\int_\T |f(t)|^{p-1}|f'(t)|dt<\infty,$$
with convention that the integrand is 0 whenever $f(t)=0$, then  Eq.(\ref{pKF}) holds. 
\end{theorem}

\begin{proof} Let $A=T_f^*T_f$, $B=T_fT_f^*$, and $q=p/2$. 
Then 
$$A-B=-H_f^*H_f+H_{\overline{f}}^*H_{\overline{f}}\in S_q(H^2)$$
thanks to the Peller theorem. 
Thus we get $A^q-B^q\in S_q(H^2)$, and  
$$\Tr(A^q-B^q)=\frac{\sin(q\pi)}{\pi}\int_0^\infty \Tr\left(\frac{1}{\lambda+B}-\frac{1}{\lambda+A}\right)\lambda^qd\lambda.$$ 
Theorem \ref{HHtoKF} implies 
\begin{align*}
\Tr(A^q-B^q)&=\frac{\sin(q\pi)}{\pi}\int_0^\infty \frac{1}{2\pi i}\int_\T \left(\frac{1}{\lambda}-\frac{1}{\lambda +|f(t)^2|}\right)
\frac{f'(t)}{f(t)}dt\lambda^qd\lambda \\
 &=\frac{\sin(q\pi)}{\pi}\frac{1}{2\pi i}\int_0^\infty \int_\T \varphi_\lambda(|f(t)|^2)\lambda^{q-1}\frac{f'(t)}{f(t)}dtd\lambda.
\end{align*}
Now the second statement follows from the Fubini theorem. 
\end{proof}

\begin{cor} Eq.(\ref{ptr}) holds for every $f\in C^\infty(\T)$ and every $p>0$. 
If moreover 
$$\int_\T |f(t)|^{p-1}|f'(t)|dt<\infty,$$
with convention that the integrand is 0 whenever $f(t)=0$, Eq.(\ref{pKF}) holds. 
\end{cor}

\begin{proof} The statement follows from 
$$C^\infty(\T)\subset \bigcap_{p>0}B_p^{1/p}(\T).$$
\end{proof}

\begin{cor} Assume $f\in B_p^{1/p}(\T)$ for all $p>0$. 
Then under the assumption of Theorem \ref{Witten}, the following limit exists: 
$$\lim_{p\to +0}\Tr(|T_f|^p-|T_f^*|^p)=-\ind_W T_f. $$ 
\end{cor}

\begin{proof} In a similar way as in the proof of Theorem \ref{Witten}, we can show
$$\lim_{p\to +0}\Tr(|T_f|^p-|T_f^*|^p)
=\lim_{p\to +0}\frac{1}{2\pi i}\int_\T |f(t)|^p\frac{f'(t)}{f(t)}dt
=\frac{1}{2\pi i}\mathrm{p.v.}\int_\T \frac{f'(t)}{f(t)}dt,$$
and the statement follows from Theorem \ref{Witten}.
\end{proof}

Recall that if a symbol $f$ is analytic, we have $T_f^*T_f-T_fT_f^*=H_{\overline{f}}^*H_{\overline{f}}\geq 0$.

\begin{theorem}\label{apthm} 
Let $f\in B_2\cap H^\infty$, and let $F(z)=\tf(z)$ be the holomorphic extension of $f$ to $\D$. 
Then for every operator monotone function $\varphi$ on $[0,\infty)$, we have 
\begin{align*}
 \Tr(\varphi(T_f^*T_f)-\varphi(T_fT_f^*))&= \frac{1}{\pi}\int_\D \varphi'(|F(z)|^2)|F'(z)|^2dA(z).\\
\end{align*}
In particular, for every $p>0$,
\begin{align}\label{pint}
\Tr(|T_f|^p-|T_f^*|^p) &=\frac{p}{2\pi}\int_\D |F(z)|^{p-2}|F'(z)|^2dA(z) \\
 &= \lim_{r\to 1-0}\frac{1}{2\pi i}\int_{|z|=r}|F(z)|^p\frac{dF(z)}{F(z)}. \notag
\end{align}
\end{theorem}

\begin{proof} 
From Eq.(\ref{trint}), we have 
\begin{align*}
\lefteqn{\Tr(\varphi(T_f^*T_f)-\varphi(T_f^*T_f))} \\
 &=a_\varphi\Tr([T_f^*,T_f])+\int_0^\infty
\Tr\left(\frac{1}{\lambda+T_fT_f^*}-\frac{1}{\lambda+T_f^*T_f}\right)\lambda^2 
d\mu_\varphi(\lambda),
\end{align*}
and Theorem \ref{AKF} implies that it is equal to 
$$\frac{1}{\pi}\int_\D a_\varphi |F'(z)|^2dA(z)+\int_0^\infty\frac{1}{\pi}\int_\D
\frac{1}{(\lambda+|F(z)|^2)^2}|F'(z)|^2dA(z)\lambda^2 d\mu_\varphi(\lambda).$$
Since 
$$\varphi'(x)=a_\varphi+\int_0^\infty \frac{1}{(\lambda+x)^2}\lambda^2 d\mu_\varphi(x),$$
the first statement follows from the Fubini theorem. 

Let $a\in \D$ be a zero of $F$ of order $n$. 
Since $F(z)$ is of the form $(z-a)^nG(z)$ with a holomorphic function $G(z)$ satisfying $G(a)\neq 0$, we get
$$\lim_{\varepsilon\to+0}\int_{|z-a|\leq \varepsilon}|F(z)|^{p-2}|F'(z)|^2dA(z)=0,$$
$$\lim_{\varepsilon\to+0}\int_{|z-a|=\varepsilon}|F(z)|^p\frac{dF(z)}{F(z)}=0.$$
Thus for $0<r<1$, the Stokes theorem implies 
$$\frac{p}{2\pi} \int_{|z|\leq r} |F(z)|^{p-2} |F'(z)|^2dA(z)
=\frac{1}{2\pi i}\int_{|z|=r}|F(z)|^p\frac{dF(z)}{F(z)},$$
and the second statement follows. 
\end{proof}

\begin{remark} For all functions in Example \ref{ExAnBe},(1),(2),(3), and for all $p>0$, the integral 
in Eq.(\ref{pint}) converges. 
Thus for such a symbol, Eq.(\ref{ptr}) holds for all $p>0$.  
This means that Theorem \ref{pthm} is not at all sharp: there are plenty of examples of analytic symbols $f$ 
such that Eq.(\ref{ptr}) holds for all $p>0$ while $f\notin B_p$ for some $p>0$. 
\end{remark}

\begin{problem} It is an interesting problem to characterize the class of holomorphic functions $F$, for which 
the integral in Eq.(\ref{pint}) converges. 
Does it converge for all $F\in B_1$ and all $p>0$ ? 
\end{problem}
 
We finish this paper with examples of explicit computations for fun.  

\begin{example} Let $S$ be the unilateral shift, and let $a,p>0$.  
Since $S+a=T_{e^{it}+a}$, we have $|S+a|^p-|S^*+a|^p\in S_1(H^2)$, and 
$$\Tr(|S+a|^p-|S^*+a|^p)=\frac{1}{2\pi i}\int_{|z|=1} \frac{|z+a|^p}{z+a}dz.$$
For $a=1$, this can be evaluated as 
$$\Tr(|S+1|^p-|S^*+1|^p)=\frac{2^p}{\pi}\int_0^{\frac{\pi}{2}}\cos^p\theta d\theta
=\frac{\Gamma(1+p)}{2\Gamma(1+\frac{p}{2})^2}.$$
For $p=1$ and $a<1$, we can compute it by the parametrization $z=-a+r(\theta)e^{i\theta}$ with 
$r(\theta)=a\cos\theta+\sqrt{1-a^2\sin^2\theta}$, and it turns out to be the elliptic integral  
$$\Tr(|S+a|-|S^*+a|)=\frac{2}{\pi}\int_0^{\frac{\pi}{2}}\sqrt{1-a^2\sin^2\theta}d\theta
=\frac{1}{\pi}\int_0^1\sqrt{\frac{1-a^2x}{x(1-x)}}dx.$$
For $p=1$ and $a>1$, 
$$\Tr(|S+a|-|S^*+a|)=\frac{2}{\pi}\int_0^{\sin^{-1}\frac{1}{a}}\sqrt{1-a^2\sin^2\theta}d\theta
=\frac{1}{\pi}\int_0^1\sqrt{\frac{1-x}{x(a^2-x)}}dx.$$

The above computation with $a=1$ can be generalized as follows. 
Note that if $F$ is a polynomial of degree $n$ whose zeros are on the unit circle, the real part 
of $e^{it}F'(e^{it})/F(e^{it})$ is $\frac{n}{2}$.  
This implies that for $f(t)=F(e^{it})$, we have 
$$\Tr\left(|T_f|^p-|T_f^*|^p\right)=\frac{1}{2\pi}\int_0^{2\pi}|F(e^{it})|^p\frac{n}{2}dt.$$
In particular, 
\begin{align*}
\lefteqn{\Tr\left(\left|\sum_{k=0}^{n-1} S^k\right|-\left|\sum_{k=0}^{n-1} {S^*}^k\right|\right)
=\frac{n-1}{4\pi}\int_0^{2\pi}\left|\sum_{k=0}^{n-1}e^{ikt}\right|dt=\frac{n-1}{4\pi}\int_0^{2\pi}\left|\frac{\sin\frac{nt}{2}}{\sin\frac{t}{2}}\right|dt} \\
 &=\frac{n-1}{4\pi}\sum_{k=0}^{n-1}(-1)^k\int_{\frac{2\pi k}{n}}^{\frac{2\pi(k+1)}{n}}\frac{\sin\frac{nt}{2}}{\sin\frac{t}{2}}dt
 =\frac{n-1}{4\pi}\sum_{k=0}^{n-1}(-1)^k\int_{\frac{2\pi k}{n}}^{\frac{2\pi(k+1)}{n}}\sum_{l=0}^{n-1}e^{i(l-\frac{n-1}{2})t}dt.
\end{align*}
For even $n$, this is 
$$\frac{n-1}{\pi}\sum_{j=0}^{\frac{n-2}{2}}\frac{\tan\frac{(j+\frac{1}{2})\pi}{n}}{j+\frac{1}{2}}.$$
For odd $n$, it is
$$\frac{n-1}{2n}+\frac{n-1}{\pi}\sum_{j=1}^{\frac{n-1}{2}}\frac{ \tan\frac{j\pi}{n}}{j}.$$
\end{example}

\section*{Acknowledgments}
The author would like to thank Fumio Hiai and Hideki Kosaki for stimulating discussions and useful comments on trace inequalities,   
and Xiang Tang for answering a question on Toeplitz operators on the weighted Bergman spaces.  

\end{document}